\documentclass{amsart}

\usepackage{amsfonts}
\usepackage[colorlinks=true,linkcolor=black,citecolor=black]{hyperref}

\theoremstyle{plain}
  \newtheorem{Thm}{Theorem}
  
  \newtheorem{Cor}[Thm]{Corollary}
  \newtheorem{Lem}{Lemma}

\theoremstyle{definition}

\theoremstyle{remark}
  \newtheorem{Rem}{Remark}

\newcommand{\proj}{\operatorname{proj}} 

\begin{document}


\title[A Hardy inequality for ultraspherical expansions]{A Hardy inequality for ultraspherical expansions with an application to the sphere}

\author[A. Arenas]{Alberto Arenas}
\address{Departamento de Matem\'aticas y Computaci\'on, Universidad de La Rioja, Complejo Cient\'{\i}fico-Tecnol\'ogico, Calle Madre de Dios 53, 26006, Logro\~no, Spain}
\email{alarenas@unirioja.es}

\author[\'O. Ciaurri]{\'Oscar Ciaurri}
\address{Departamento de Matem\'aticas y Computaci\'on, Universidad de La Rioja, Complejo Cient\'{\i}fico-Tecnol\'ogico, Calle Madre de Dios 53, 26006, Logro\~no, Spain}
\email{oscar.ciaurri@unirioja.es}
\thanks{Research of the second author supported by grant MTM2015-65888-C4-4-P of the Spanish government}

\author[E. Labarga]{Edgar Labarga}
\address{Departamento de Matem\'aticas y Computaci\'on, Universidad de La Rioja, Complejo Cient\'{\i}fico-Tecnol\'ogico, Calle Madre de Dios 53, 26006, Logro\~no, Spain}
\email{edlabarg@unirioja.es}

\begin{abstract}
We prove a Hardy inequality for ultraspherical expansions by using a proper ground state representation. From this result we deduce some uncertainty principles for this kind of expansions. Our result also implies a Hardy inequality on spheres with a potential having a double singularity.
\end{abstract}

\maketitle

\section{Introduction and main result}

For $d\ge 3$, the classical Hardy inequality states that
\begin{equation}
\label{ec:hardy}
\frac{(d-2)^2}{4}\int_{\mathbb{R}^d}\frac{u^2(x)}{|x|^2}\, dx\le
\int_{\mathbb{R}^d}|\nabla u(x)|^2\, dx.
\end{equation}
Due to its applicability, there is an extensive literature about the topic (see the references in \cite{RT}) covering many extensions of this estimate in several and different directions. We are interested in one involving the fractional powers of the Laplacian. We can rewrite \eqref{ec:hardy} as
\begin{equation*}
\frac{(d-2)^2}{4}\int_{\mathbb{R}^d}\frac{u^2(x)}{|x|^2}\, dx\le \int_{\mathbb{R}^d} u(x) (-\Delta u(x)) \, dx
\end{equation*}
and, taking the fractional Laplacian $(-\Delta)^\sigma$ defined by $\widehat{(-\Delta)^\sigma u}=|\cdot|^{2\sigma} \widehat{u}$, a natural extension is the inequality
\begin{equation}
\label{ec:Hardy-frac}
C_{\sigma,d}\int_{\mathbb{R}^d}\frac{u^2(x)}{|x|^{2\sigma}}\, dx\le \int_{\mathbb{R}^d} u(x) (-\Delta)^{\sigma} u(x) \, dx,
\end{equation}
for which the sharp constant $C_{\sigma,d}$ is well known (see \cite{Beckner3,Yafaev}).

From \eqref{ec:Hardy-frac}, we deduce the positivity (in a distributional sense) of the operator
\[
(-\Delta)^\sigma-\frac{C_{\sigma,d}}{|\cdot|^{2\sigma}}.
\]
Our target is to provide a Hardy inequality like \eqref{ec:Hardy-frac} related to ultraspherical expansions and apply it to prove the positivity of certain operator on the sphere with a potential having singularities in both poles of the sphere.

Let $C_n^\lambda(x)$ be the ultraspherical polynomial of degree $n$ and order $\lambda>-1/2$. We consider $c_n^\lambda(x)=d_n^{-1}C_n^\lambda(x)$ with
\[
d_n^2=
\int_{-1}^{1}\left(C_n^\lambda(x)\right)^2\,d\mu_\lambda(x),\qquad
d\mu_\lambda(x)=(1-x^2)^{\lambda-1/2}\, dx.
\]
The sequence of polynomials $\{c_n^\lambda\}_{n\ge 0}$ forms an orthonormal basis of the space $L_\lambda^2:=L^2((-1,1),d\mu_\lambda)$. For each $c_n^\lambda$, it holds that $\mathcal{L}_\lambda c_n^\lambda=-(n+\lambda)^2 c_n^\lambda$, where
\[
\mathcal{L}_\lambda =(1-x^2)\frac{d^2}{dx^2}-(2\lambda+1)x\frac{d}{dx}-\lambda^2.
\]
The ultraspherical expansion of each appropriate function $f$ defined in $(-1,1)$ is given by
\[
f\longmapsto \sum_{n=0}^\infty a_n^\lambda(f)c_n^\lambda,
\]
where $a_n^\lambda(f)$ is the $n$-th Fourier coefficient of $f$ respect to $\{c_n^\lambda\}_{n\ge 0}$, i.e.,
\[
a_n^\lambda(f)=\int_{-1}^{1}f(y)c_n^\lambda(y)\,d\mu_\lambda(y).
\]

The fractional powers of the operator $\mathcal{L}_\lambda$ are defined by
\[
(-\mathcal{L}_\lambda)^{\sigma/2}f=\sum_{n=0}^\infty(n+\lambda)^{\sigma} a_n^\lambda(f)c_n^\lambda, \qquad \sigma >0.
\]
This operator should be the natural candidate to prove a Hardy type inequality for the ultraspherical expansion but, however, it is not the most appropriate in this setting. We have to consider other one with an analogous behaviour to $(-\mathcal{L}_\lambda)^{\sigma/2}$, in order to deduce some results on the sphere. For each $\sigma>0$ we define (spectrally) the operator
\[
A_\sigma^{\lambda} =\frac{\Gamma(\sqrt{-\mathcal{L}_\lambda}+\frac{1+\sigma}{2})}
{\Gamma(\sqrt{-\mathcal{L}_\lambda}+\frac{1-\sigma}{2})}.
\]
Then for $f$ defined on the interval $(-1,1)$
$$
A_\sigma^{\lambda} f(x)=\sum_{n=0}^{\infty}
\frac{\Gamma(n+\lambda+\frac{1+\sigma}{2})}{\Gamma(n+\lambda+\frac{1-\sigma}{2})}a_n^\lambda(f)c_n^\lambda(x).
$$
Note that
\begin{equation}
\label{eq:asym}
\frac{\Gamma(n+\lambda+\frac{1+\sigma}{2})}{\Gamma(n+\lambda+\frac{1-\sigma}{2})}\simeq (n+\lambda)^{\sigma},
\end{equation}
then the behaviour of $(-\mathcal{L}_\lambda)^{\sigma/2}$ and $A_\sigma^\lambda$ is similar. The natural Sobolev space to analyse Hardy type inequalities is
\[
H^{\sigma}_\lambda=\Big\{f\in L^2_\lambda : \|f\|_{H^{\sigma}_\lambda}:=\Big(\sum_{n=0}^\infty (n+\lambda)^{\sigma}(a_n^\lambda(f))^2\Big)^{1/2}<\infty\Big\}.
\]
We have to note that $H^\sigma_\lambda$ is equivalent to the space $\mathcal{L}_{\lambda,\sigma}^2$ introduced in \cite{Betancoretal}.

With the previous notation our Hardy inequality for ultraspherical expansions is given in the following result.
\begin{Thm}\label{theorem01}
Let $\lambda>0$ and $0<\sigma<1$. Then for $u\in H_\lambda^\sigma$
\begin{equation}
\label{ec:Hardy-Ultra}
  Q_{\sigma,\lambda}\int_{-1}^{1}\frac{u^2(x)}{(1-x^2)^{\sigma/2}}\,d\mu_\lambda(x)\leq
  \int_{-1}^{1}u(x)A_\sigma^{\lambda} u(x)\,d\mu_\lambda(x),
\end{equation}
where
\begin{equation}
\label{ec:cons-Hardy}
  Q_{\sigma,\lambda}=2^\sigma
  \frac{\Gamma(\frac{\lambda}{2}+\frac{1+\sigma}{4})^2}{\Gamma(\frac{\lambda}{2}+\frac{1-\sigma}{4})^2}.
\end{equation}
\end{Thm}

Inequality \eqref{ec:Hardy-Ultra} can be rewritten in terms of the Fourier coefficients
\begin{equation}
\label{ec:Pitt}
 Q_{\sigma,\lambda}\int_{-1}^{1}\frac{u^2(x)}{(1-x^2)^{\sigma/2}}\,d\mu_\lambda(x)\leq
 \sum_{n=0}^\infty \frac{\Gamma(n+\lambda+\frac{1+\sigma}{2})}{\Gamma(n+\lambda+\frac{1-\sigma}{2})}(a_n^\lambda(u))^2,
\end{equation}
which is a kind of Pitt inequality for the ultraspherical expansions (for other Pitt inequalities see \cite{Beckner2,GIT}). Note that for the right hand side of \eqref{ec:Hardy-Ultra} we have, by \eqref{eq:asym},
\[
\int_{-1}^{1}u(x)A_\sigma^{\lambda} u(x)\,d\mu_\lambda(x)=\sum_{n=0}^\infty \frac{\Gamma(n+\lambda+\frac{1+\sigma}{2})}{\Gamma(n+\lambda+\frac{1-\sigma}{2})}(a_n^\lambda(u))^2
\simeq \|u\|_{H^{\sigma}_\lambda}^2,
\]
so the space $H_\lambda^\sigma$ is the adequated one.

The proof of Theorem \ref{theorem01} will be a consequence of a proper ground state representation in our setting, analogous to the given one in the Euclidean case in \cite{FLS}. Following the ideas in that paper, we can see that the constant $Q_{\sigma,\lambda}$ is sharp but not achieved. Similar ideas have been recently exploited in \cite{CRT,RT}.

From \eqref{ec:Hardy-Ultra}, by using Cauchy–-Schwarz inequality, we can obtain a Heisenberg type uncertainty principle as it was done for the sublaplacian of the Heisenberg group in \cite{GL}, and for the fractional powers of the same sublaplacian in~\cite{RT}.

\begin{Cor}
Let $\lambda>0$ and $0<\sigma<1$. Then for $u\in H_\lambda^\sigma$
\begin{equation*}
  Q_{\sigma,\lambda}\left(\int_{-1}^{1} u^2(x)\,d\mu_\lambda(x)\right)^{2}\leq
  \int_{-1}^{1}u^2(x)(1-x^2)^{\sigma/2}\,d\mu_\lambda(x)
  \int_{-1}^{1}u(x)A_\sigma^{\lambda} u(x)\,d\mu_\lambda(x),
\end{equation*}
where $Q_{\sigma,\lambda}$ is the constant given in \eqref{ec:cons-Hardy}.
\end{Cor}

Pitt inequality \eqref{ec:Pitt} allows us to prove a logarithmic uncertainty principle for the ultraspherical expansions. The main idea comes from \cite{Beckner3}. By an elementary argument, for a derivable function such that $\phi(0)=0$ and $\phi(\sigma)>0$ for $\sigma\in(0,\varepsilon)$, with $\varepsilon >0$, it is verified that $\phi'(0_+)\ge 0$. Then, taking the function
\[
\phi(\sigma)=
 \sum_{n=0}^\infty \frac{\Gamma(n+\lambda+\frac{1+\sigma}{2})}{\Gamma(n+\lambda+\frac{1-\sigma}{2})}(a_n^\lambda(u))^2
 -Q_{\sigma,\lambda}\int_{-1}^{1}\frac{u^2(x)}{(1-x^2)^{\sigma/2}}\,d\mu_\lambda(x),
\]
we have $\phi(0)=0$ (this is Parseval identity) and, by \eqref{ec:Pitt}, $\phi(\sigma)>0$ for $\sigma \in (0,1)$, then $\phi'(0_+)\ge 0$ and this inequality gives the logarithmic uncertainty principle, which is written as
\begin{multline*}
\left(\log 2+\psi\left(\frac{\lambda}{2}+\frac14\right)\right)\int_{-1}^1 u^2(x)\, d\mu_\lambda(x)\\\le
\sum_{n=0}^\infty \psi\left(n+\lambda+\frac12\right)(a_n(u))^2+\int_{-1}^1\log(\sqrt{1-x^2})u^2(x)\, d\mu_\lambda(x),
\end{multline*}
where $\psi(a)=\frac{\Gamma'(a)}{\Gamma(a)}$.

In next section we will show an application of Theorem \ref{theorem01} to obtain a Hardy inequality on the sphere. The results in Section \ref{sec:aux} are the main ingredients in the proof of Theorem \ref{theorem01} which is given in last section of the paper.

\section{An application to the sphere}

It is well known that $L^2(\mathbb{S}^d)=\oplus_{n=0}^\infty \mathcal{H}_n(\mathbb{S}^d)$, where $\mathcal{H}_n(\mathbb{S}^d)$ is  the set of spherical harmonics of degree $n$ in $d+1$ variables. If we consider the shifted Laplacian on the sphere
\[
-\Delta_{\mathbb{S}^d}=\tilde{-\Delta_{\mathbb{S}^d}}+\left(\frac{d-1}{2}\right)^2,
\]
where $\tilde{-\Delta_{\mathbb{S}^d}}$ is the Laplace-Beltrami operator on $\mathbb{S}^d$, it is verified that
\[
-\Delta_{\mathbb{S}^d} \mathcal{H}_n(\mathbb{S}^d)=\left(n+\frac{d-1}{2}\right)^2\mathcal{H}_n(\mathbb{S}^d).
\]
In this way, the analogous of the operator $A_\sigma^{\lambda}$ on $\mathbb{S}^d$ is defined by
\begin{align*}
\mathbf{A}_\sigma f&=\frac{\Gamma\left(\sqrt{-\Delta_{\mathbb{S}^d}}+\frac{1+\sigma}{2}\right)}
{\Gamma\left(\sqrt{-\Delta_{\mathbb{S}^d}}+\frac{1-\sigma}{2}\right)}f\\
&=\sum_{n=0}^\infty\frac{\Gamma\left(n+\frac{d-1}{2}+\frac{1+\sigma}{2}\right)}
{\Gamma\left(n+\frac{d-1}{2}+\frac{1-\sigma}{2}\right)}\proj_{\mathcal{H}_n(\mathbb{S}^d)}f,
\end{align*}
where $\proj_{\mathcal{H}_n(\mathbb{S}^d)}f$ denotes the projection of $f$ onto the eigenspace $\mathcal{H}_n(\mathbb{S}^d)$.

The operator $\mathbf{A}_\sigma$ becomes the fractional powers of the Laplacian in the Euclidean space through conformal transforms as was observed by T. P. Branson in~\cite{Branson}. So $\mathbf{A}_\sigma$ is the natural operator to prove a Hardy type inequality on the sphere. In our proof, we will write $\mathbf{A}_\sigma$ in terms of $A_\sigma^\lambda$ and this is the main reason to consider $A_\sigma^\lambda$ in the case of the ultraspherical expansions. An analogous of the Hardy-Littlewood-Sobolev inequality for $\mathbf{A}_\sigma$ and some other inequalities for it were given by W. Beckner in \cite{Beckner-Annals}. The operators $\mathbf{A}_\sigma$ also appear in \cite[p. 151]{Samko} and \cite[p. 525]{Rubin}.

Each point $x\in \mathbb{S}^d$ can be written as
\[
x=(t,\sqrt{1-t^2}x'_1,\dots,\sqrt{1-t^2}x'_d),
\]
for $t\in (-1,1)$ and $x':=(x'_1,\dots,x'_d)\in \mathbb{S}^{d-1}$, and so
\[
\int_{\mathbb{S}^d} f(x)\, dx=\int_{-1}^1 \int_{\mathbb{S}^{d-1}}f(t,\sqrt{1-t^2}x')(1-t^2)^{(d-2)/2}\, dx'\, dt.
\]
With these coordinates, see \cite[Section 3]{Sherman}, we have that an orthonormal basis for each $\mathcal{H}_n(\mathbb{S}^d)$ is given by
\[
\phi_{n,j,k}(x)=\psi_{n,j}(t)Y_{j,k}^d (x'), \qquad j=0,\dots,n,
\]
with
\[
\psi_{n,j}(t)=(1-t^2)^{j/2}c_{n-j}^{j+(d-1)/2}(t)
\]
and $\{Y_{j,k}^d\}_{k=1,\dots,d(j)}$ an orthonormal basis of spherical harmonics on $\mathbb{S}^{d-1}$ of degree $j$. The value $d(j)$ indicates the dimension of $\mathcal{H}_j(\mathbb{S}^{d-1})$; i.e.,
\[
d(j)=(2j+d-2)\frac{(j+d-3)!}{j!(d-2)!}.
\]
Then, the orthogonal projection of $f$ onto the eigenspace $\mathcal{H}_n(\mathbb{S}^d)$ can be written as
\[
\proj_{\mathcal{H}_n(\mathbb{S}^d)}f=\sum_{j=0}^n \sum_{k=1}^{d(j)} f_{n,j,k} \phi_{n,j,k},
\]
with
\[
f_{n,j,k}=\int_{-1}^{1} G_{j,k}(t) c_{n-j}^{j+(d-1)/2}(t) (1-t^2)^{j+(d-2)/2}\, dt,
\]
\[
G_{j,k}(t)=(1-t^2)^{-j/2}F_{j,k}(t)\quad\text{ and }\quad
F_{j,k}(t)=\int_{\mathbb{S}^{d-1}}f(t,\sqrt{1-t^2}x')Y_{j,k}^d(x')\, dx'.
\]
It is easy to observe that
\[
f(x)=\sum_{j=0}^\infty \sum_{k=1}^{d(j)} F_{j,k}(t)Y_{j,k}^d(x')= \sum_{j=0}^\infty \sum_{k=1}^{d(j)} (1-t^2)^{j/2}G_{j,k}(t)Y_{j,k}^d(x').
\]
Moreover, from the definition of $\mathbf{A}_\sigma$, we have
\[
\mathbf{A}_\sigma f(x)=\sum_{j=0}^\infty \sum_{k=1}^{d(j)}(1-t^2)^{j/2} A_{\sigma}^{j+(d-1)/2} G_{j,k}(t)Y_{j,k}^d(x').
\]

Now, considering the Sobolev space
\[
\mathbf{H}^{\sigma}=\Big\{f\in L^2(\mathbb{S}^d) : \|f\|_{\mathbf{H}^{\sigma}}:=\Big(\sum_{n=0}^\infty \Big(n+\frac{d-1}{2}\Big)^{\sigma}
\|\proj_{\mathcal{H}_n(\mathbb{S}^d)}f\|^2_{L^2(\mathbb{S}^d)}\Big)^{1/2}<\infty\Big\},
\]
we have the following Hardy inequality on the sphere.
\begin{Thm}\label{theorem02}
Let $d\ge 2$, $0<\sigma<1$, and $e_d$ be the north pole of the sphere $\mathbb{S}^d$. Then for $f\in \mathbf{H}^\sigma$
\begin{equation}
\label{ec:Hardy-sphere}
  2^\sigma Q_{\sigma,(d-1)/2}\int_{\mathbb{S}^d}\frac{f^2(x)}{(|x-e_d||x+e_d|)^{\sigma}}\,dx\leq
  \int_{\mathbb{S}^d}f(x)\mathbf{A}_\sigma f(x)\,dx,
\end{equation}
where $Q_{\sigma,(d-1)/2}$ is the constant given in \eqref{ec:cons-Hardy}.
\end{Thm}

\begin{proof}
By the orthogonality of the spherical harmonics, it is elementary to show that
\[
\int_{\mathbb{S}^d}f(x)\mathbf{A}_\sigma f(x)\,dx=\sum_{j=0}^\infty \sum_{k=1}^{d(j)}\int_{-1}^1 G_{j,k}(t) A_{\sigma}^{j+(d-1)/2} G_{j,k}(t)\, d\mu_{j+(d-1)/2}(t).
\]
Now, applying Theorem \ref{theorem01}, we deduce that
\[
\int_{\mathbb{S}^d}f(x)\mathbf{A}_\sigma f(x)\,dx\ge
\sum_{j=0}^\infty \sum_{k=1}^{d(j)}Q_{\sigma, j+(d-1)/2} \int_{-1}^1 \frac{F_{j,k}^2(t)}{(1-t^2)^{\sigma/2}} \, d\mu_{(d-1)/2}(t).
\]
It is known (see \cite{Yafaev}) that for $0<x\le y$ and $j\ge 0$ we have that $\frac{\Gamma(j+y)}{\Gamma(j+x)}\ge \frac{\Gamma(y)}{\Gamma(x)}$. So, $Q_{\sigma,j+(d-1)/2}\ge Q_{\sigma,(d-1)/2}$ and
\[
\int_{\mathbb{S}^d}f(x)\mathbf{A}_\sigma f(x)\,dx\ge Q_{\sigma,(d-1)/2}
\sum_{j=0}^\infty \sum_{k=1}^{d(j)}\int_{-1}^1 \frac{F_{j,k}^2(t)}{(1-t^2)^{\sigma/2}} \, d\mu_{(d-1)/2}(t).
\]
The proof of \eqref{ec:Hardy-sphere} is finished by using the identity
\[
\sum_{j=0}^\infty \sum_{k=1}^{d(j)}\int_{-1}^1 \frac{F_{j,k}^2(t)}{(1-t^2)^{\sigma/2}} \, d\mu_{(d-1)/2}(t)=
2^\sigma \int_{\mathbb{S}^d}\frac{f^2(x)}{(|x-e_d||x+e_d|)^{\sigma}}\, dx.
\]
\end{proof}

The analogous role on the sphere of radially symmetric functions is played by functions which are invariant under the action of $SO(d-1)$. By $SO(d-1)$-invariance we mean that $f$ is invariant under the action of the group $SO(d-1)$ on $\mathbb{S}^{d-1}$ whenever $SO(d-1)$ is embedded into $SO(d)$ in a suitable way. Each function $f$ of this kind can be written as $f(x)=g(\langle x,e_d\rangle )$, for a certain function $g$ defined in $(-1,1)$. Then for this kind of functions Theorem \ref{theorem02} reduces to Theorem \ref{theorem01} with $\lambda=(d-1)/2$, in this way we can deduce that the constant $2^\sigma Q_{\sigma,(d-1)/2}$ in \eqref{ec:Hardy-sphere} is sharp.

As in the classic case, from Theorem \ref{theorem02} we deduce that in a distributional sense
\[
\mathbf{A}_\sigma-\frac{2^\sigma Q_{\sigma,(d-1)/2}}{(|x-e_d||x+e_d|)^{\sigma}}\ge 0.
\]
Note that in this case we are perturbing the operator $\mathbf{A}_\sigma$ adding a potential with singularities in both poles of the sphere.

\section{Auxiliary results}
\label{sec:aux}

The following lemmas give the tools to prove Theorem \ref{theorem01}. To be more precise, Lemma \ref{lemma01} provides a nonlocal representation of the operator $A_\sigma^{\lambda}$ with a kernel having nice properties for our target. Lemma \ref{lemma02} shows the action of the operator $A_\sigma^{\lambda}$ on the family of weights $(1-x^2)^{-(\lambda/2+(1-\sigma)/4)}$.

For $f,g\in L_\lambda^2$ we are going to set up the  notation
$$
\langle f,g\rangle_\lambda=\int_{-1}^{1}f(x)g(x)\,d\mu_\lambda(x)
$$
to simplify the writing.

\begin{Lem}\label{lemma01}
Let $\lambda>0$ and $0<\sigma<1$. If $f$ is a finite linear combination of ultraspherical polynomials, then
\begin{equation}\label{eq:Lem01_principal}
  A_\sigma^{\lambda} f(x)
  =
  \int_{-1}^1\left(f(x)-f(y)\right)K_\sigma^\lambda(x,y)\,d\mu_\lambda(y)+E_{\sigma,\lambda}f(x), \qquad x\in (-1,1),
\end{equation}
where the kernel is given by
\[
  K_\sigma^\lambda(x,y)
  =
  D_{\sigma,\lambda}\int_{-1}^{1}\frac{d\mu_{\lambda-1/2}(t)}{(1-xy-\sqrt{1-x^2}\sqrt{1-y^2}t)^{\lambda+(1+\sigma)/2}},
\]
with
\[
  D_{\sigma,\lambda}=\frac{c_\lambda^2}{2^{\lambda+(1+\sigma)/2}}
  \frac{\Gamma(\frac{1-\sigma}{2})\Gamma(\lambda+\frac{1+\sigma}{2})}{|\Gamma(-\sigma)|\Gamma(1+\lambda)}, \qquad   c_\lambda=\frac{\Gamma(2\lambda+1)}{2^{2\lambda}(\Gamma(\lambda+1/2))^2},
\]
and
\[
  E_{\sigma,\lambda}=\frac{\Gamma(\lambda+\frac{1+\sigma}{2})}{\Gamma(\lambda+\frac{1-\sigma}{2})}.
\]

Moreover, for $f\in H_\lambda^\sigma$ we have
\begin{equation}
\label{eq:esc-prod}
\langle A_\sigma^\lambda f, f\rangle_\lambda=\frac{1}{2}\int_{-1}^1\int_{-1}^1 (f(x)-f(y))^2 K_\sigma^\lambda (x,y) \, d\mu_\lambda(y)\, d \mu_\lambda(x)+E_{\sigma,\lambda}\langle f, f \rangle_\lambda
\end{equation}
\end{Lem}

\begin{proof}
  We start with the identity
  \begin{equation}\label{eq:Lema0}
    \int_{0}^{\infty}\left(e^{-(n+\lambda)t}-e^{-(\sigma-1)t/2}\right)\left(\sinh t/2\right)^{-\sigma-1}\,dt
    =
    2^{1+\sigma}\Gamma(-\sigma)\frac{\Gamma(n+\lambda+\frac{1+\sigma}{2})}{\Gamma(n+\lambda+\frac{1-\sigma}{2})}
  \end{equation}
  for $\lambda>0$ (actually it is also true for values $\lambda>-1/2$) and $0<\sigma<1$. To deduce the previous identity it is enough to apply integration by parts with $u=e^{-(n+\lambda+(1-\sigma)/2)t}-1$ and $v=-2e^{-\sigma t/2}(\sinh t/2)^{-\sigma}/\sigma$, and use \cite[eq. 8, p. 367]{PrudnikovI}
  $$
    \int_{0}^{\infty}e^{-\rho t}\left(\cosh (ct)-1\right)^{\nu}\,dt=\frac{\Gamma(\frac{\rho}{c}-\nu)\Gamma(2\nu+1)}{2^\nu c\Gamma(\frac{\rho}{c}+\nu+1)}
  $$
  for $c>0$, $2\nu>-1$, and $\rho>c\nu$.

  Now, we consider the Poisson operator for ultraspherical expansions. It is given by
  \[
  e^{-t\sqrt{-\mathcal{L}_\lambda}}f(x)=\sum_{n=0}^{\infty}e^{-(n+\lambda)t}a_n^\lambda(f)c_n^\lambda(x)=\int_{-1}^1 f(y)P_t^\lambda(x,y)\, d\mu_\lambda(y),
  \]
  with
  \[
  P_t^\lambda(x,y)=\sum_{n=0}^\infty e^{-(n+\lambda)t}c_n^\lambda(x)c_n^\lambda(y).
  \]
  By the product formula for ultraspherical polynomials \cite[eq. B.2.9, p. 419]{DaiXu}
  \[
    \frac{C_n^\lambda(x)C_n^\lambda(y)}{C_n^\lambda(1)}=c_\lambda\int_{-1}^{1}
    C_n^\lambda(xy+\sqrt{1-x^2}\sqrt{1-y^2}t)\,d\mu_{\lambda-1/2}(t), \qquad \lambda>0,
  \]
 the identity \cite[eq. B.2.8. p. 419]{DaiXu}
  \[
    \sum_{n=0}^{\infty}\frac{n+\lambda}{\lambda}C_n^\lambda(x)r^n=\frac{1-r^2}{(1-2xr+r^2)^{\lambda+1}}, \qquad 0\leq r<1,
  \]
  and the relation $d_n^2=\frac{\lambda}{c_\lambda(n+\lambda)} C_n^\lambda(1)$, we deduce the expression
  \[
  P_t^\lambda(x,y)=\frac{c_\lambda^2}{2^\lambda} \int_{-1}^{1}\frac{\sinh t}{(\cosh t -w(s))^{\lambda+1}}\,d\mu_{\lambda-1/2}(s),
  \]
  with $w(s)=xy+\sqrt{1-x^2}\sqrt{1-y^2}s$. The previous identity for $P_t^\lambda$ is not new, it appears as formula (2.12) in \cite{MS}.

  Combining \eqref{eq:Lema0} and the definition of the Poisson operator, it is clear that
  \[
    A_\sigma^{\lambda} f(x)
    =
    \frac{1}{2^{1+\sigma}\Gamma(-\sigma)}\int_{0}^{\infty}
    \left(e^{-t\sqrt{-\mathcal{L}_{\lambda}}}f(x)-f(x)e^{-(\sigma-1)t/2}\right)\left(\sinh t/2\right)^{-\sigma-1}\,dt,
  \]
  which can be splitted in
  \begin{multline}\label{eq:IntSplitLem1}
    A_\sigma^{\lambda} f(x)
    \\=
    \frac{1}{2^{1+\sigma}\Gamma(-\sigma)}\int_{0}^{\infty}
    \left(e^{-t\sqrt{-\mathcal{L}_\lambda}}f(x)-f(x)e^{-t\sqrt{-\mathcal{L}_\lambda}}1(x)\right)\left(\sinh t/2\right)^{-\sigma-1}\,dt
      \\
      +\frac{f(x)}{2^{1+\sigma}\Gamma(-\sigma)}\int_{0}^{\infty}
      \left(e^{-t\sqrt{-\mathcal{L}_\lambda}}1(x)-e^{-(\sigma-1)t/2}\right)\left(\sinh t/2\right)^{-\sigma-1}\,dt.
  \end{multline}
  From the obvious identity
  \[
    e^{-t\sqrt{-\mathcal{L}_\lambda}}1(x)=\int_{-1}^{1}P_t^\lambda(x,y)\,d\mu_\lambda(y)=e^{-\lambda t},
  \]
  for the second term in \eqref{eq:IntSplitLem1} we have
  \begin{equation*}
  \begin{split}
    \frac{f(x)}{2^{1+\sigma}\Gamma(-\sigma)}\int_{0}^{\infty}&
      \left(e^{-t\sqrt{-\mathcal{L}_\lambda}}1(x)-e^{-(\sigma-1)t/2}\right)\left(\sinh t/2\right)^{-\sigma-1}\,dt
    \\ &=
      \frac{f(x)}{2^{1+\sigma}\Gamma(-\sigma)}\int_{0}^{\infty}\left(e^{-\lambda t}-e^{-(\sigma-1)t/2}\right)\left(\sinh t/2\right)^{-\sigma-1}\,dt
    \\ &=
      E_{\sigma,\lambda}f(x),
  \end{split}
  \end{equation*}
  where we have used \eqref{eq:Lema0} with $n=0$.

  The first integral in \eqref{eq:IntSplitLem1} verifies
  \begin{multline*}
    \frac{1}{2^{1+\sigma}\Gamma(-\sigma)}\int_{0}^{\infty}
    \left(e^{-t\sqrt{-\mathcal{L}_\lambda}}f(x)-f(x)e^{-t\sqrt{-\mathcal{L}_\lambda}}1(x)\right)\left(\sinh t/2\right)^{-\sigma-1}\,dt\\
    \begin{aligned}
    &= \frac{1}{2^{1+\sigma}|\Gamma(-\sigma)|}\int_{0}^{\infty}\int_{-1}^{1} P_t^\lambda(x,y)(f(x)-f(y))\, d\mu_\lambda(y)\left(\sinh t/2\right)^{-\sigma-1}\,dt
    \\& =
    \frac{1}{2^{1+\sigma}|\Gamma(-\sigma)|}\int_{-1}^{1}\left(f(x)-f(y)\right)
    \int_{0}^{\infty}P_t^\lambda(x,y)\left(\sinh t/2\right)^{-\sigma-1}\,dt\,d\mu_\lambda(y)
    \\& =
   \int_{-1}^{1}\left(f(x)-f(y)\right)K_\sigma^\lambda(x,y)\,d\mu_\lambda(y),
    \end{aligned}
  \end{multline*}
  with
  \[
   K_\sigma^\lambda(x,y)= \frac{1}{2^{1+\sigma}|\Gamma(-\sigma)|}\int_{0}^{\infty}P_t^\lambda (x,y)\left(\sinh t/2\right)^{-\sigma-1}\,dt.
  \]
  In last computation we have used Fubini theorem. This is justified for finite combinations of ultraspherical polynomials by using the estimate
  \begin{equation*}
  P_t^\lambda(x,y)\le \frac{C \sinh t}{(1-x^2)^{\lambda/2}(1-y^2)^{\lambda/2}(\cosh t-xy-\sqrt{1-x^2}\sqrt{1-y^2})},
  \end{equation*}
  which follows from the elementary inequality
  \[
  \int_{-1}^1 \frac{(1-s^2)^{\lambda-1}}{(A-Bs)^{\lambda+1}}\, ds\le  \frac{C}{B^\lambda(A-B)}, \qquad A>B>0, \quad \lambda>0,
  \]
  and the mean value theorem.
  Indeed, taking $C_f=\max\{|f'(x)|:x\in [-1,1]\}$ and using the inequality $1-xy-\sqrt{1-x^2}\sqrt{1-y^2}\ge C|x-y|^2$, we have
  \begin{multline*}
  \int_{0}^{\infty}\int_{-1}^{1} P_t^\lambda(x,y)|f(x)-f(y)|\, d\mu_\lambda(y)\left(\sinh t/2\right)^{-\sigma-1}\,dt\\\le \frac{C_f}{(1-x^2)^{\lambda/2}} \left(C_1\int_0^1\int_{-1}^1 \frac{t^{-\sigma}|x-y|}{t^2+|x-y|^2}(1-y^2)^{\lambda/2-1/2}\, dy \, dt\right.\\\left.+ C_2\int_1^\infty\int_{-1}^1 e^{-(\sigma+1)t/2}|x-y|(1-y^2)^{\lambda/2-1/2}\, dy \, dt\right)=: \frac{C_f}{(1-x^2)^{\lambda/2}} ( I_1+I_2).
  \end{multline*}
  Obviously, $I_2$ is a finite integral. For $I_1$ the change of variable $t=|x-y|s$ gives
  \[
  I_1\le C_1 \int_0^\infty\frac{s^{-\sigma}}{s^2+1}\, ds\int_{-1}^1 |x-y|^{-\sigma}(1-y^2)^{\lambda/2-1/2}\, dy<\infty.
  \]

  To obtain the expression of $K_\sigma^\lambda$ we observe that
  \begin{multline*}
    K_\sigma^\lambda(x,y)\\
    \begin{aligned}
    &=\frac{c_\lambda^2}{2^{\lambda+1+\sigma}|\Gamma(-\sigma)|}\int_{0}^{\infty}\int_{-1}^{1}\frac{\sinh t}{(\cosh t -w(s))^{\lambda+1}}\,d\mu_{\lambda-1/2}(s)\left(\sinh t/2\right)^{-\sigma-1}\,dt
    \\ &=
    \frac{c_\lambda^2}{2^{\lambda+(1+\sigma)/2}}
    \frac{\Gamma(\frac{1-\sigma}{2})\Gamma(\lambda+\frac{1+\sigma}{2})}{|\Gamma(-\sigma)|\Gamma(\lambda+1)}
      \int_{-1}^{1}\frac{d\mu_{\lambda-1/2}(s)}{(1-w(s))^{\lambda+(1+\sigma)/2}},
    \end{aligned}
  \end{multline*}
  where we have applied Fubini theorem and the change of variable $2(\sinh t/2)^2=z(1-w(s))$ in last equality. With the last identity we have concluded the proof of \eqref{eq:Lem01_principal}.

To prove \eqref{eq:esc-prod} we follow the argument in \cite[Lemma 5.1]{RT}. First, we observe that the kernel $K_\sigma^\lambda(x,y)$ is positive and symmetric in the sense that $K_\sigma^\lambda(x,y)=K_\sigma^\lambda(y,x)$. Then, \eqref{eq:esc-prod} is clear when $f$ is a finite linear combination of ultraspherical polynomials. For $f\in H_\lambda^\sigma$ we consider a sequence of finite linear combinations of ultraspherical polynomials $\{p_k\}_{k\ge 0}$ such that $p_k$ converges to $f$ in $H_\lambda^\sigma$. Then, by using the definition of $A_\sigma^\lambda$, it is clear that $\langle A_\sigma^\lambda p_k, p_k\rangle_{\lambda}$ converges to $\langle A_\sigma^\lambda f, f\rangle_{\lambda}$. Moreover, the result for polynomial functions implies
\begin{multline}
\label{eq:esc-prod-pk}
\langle A_\sigma^\lambda p_k, p_k\rangle_{\lambda}=\frac{1}{2}\int_{-1}^1\int_{-1}^1 (p_k(x)-p_k(y))^2 K_\sigma^\lambda (x,y) \, d\mu_\lambda(y)\, d \mu_\lambda(x)\\+E_{\sigma,\lambda}\langle p_k, p_k \rangle_\lambda<\infty.
\end{multline}
Consequently, the functions $P_k(x,y)=p_k(x)-p_k(y)$ form a Cauchy sequence in $L^2((-1,1)\times (-1,1), d\omega)$ where $d\omega(x,y)=K_\sigma^\lambda(x,y)\, d\mu_\lambda(x)\, d\mu_\lambda(y)$ which converges to $f(x)-f(y)$ in this norm. Hence, passing to the limit in \eqref{eq:esc-prod-pk}, we complete the proof of the lemma.
\end{proof}

\begin{Lem}\label{lemma02}
Let $\lambda>0$ and $2\lambda+1>\sigma>0$. Then
\begin{equation}\label{eq:Lem02_principal}
A_\sigma^{\lambda}\left(\frac{1}{(1-x^2)^{\lambda/2+(1-\sigma)/4}}\right)=
\frac{Q_{\sigma,\lambda}}{(1-x^2)^{\lambda/2+(1+\sigma)/4}},
\end{equation}
where $Q_{\sigma,\lambda}$ is the constant given in \eqref{ec:cons-Hardy}.
\end{Lem}

\begin{proof}
  First of all, we have to realize that the ultraspherical polynomial $C_{n}^\lambda(x)$ is odd for $n=2m+1$, $m\in\mathbb{Z}^{+}$; therefore, for $\beta>0$, the function $(1-x^2)^{\beta-1} C_{2m+1}^\lambda(x)$ is an odd function and its integral over the interval $(-1,1)$ is zero. For $n=2m$ we use \cite[eq. 15, p. 519]{PrudnikovII} to obtain
  \begin{equation*}
  \begin{split}
    \int_{-1}^{1}&(1-x^2)^{\beta-1}C_{2m}^\lambda(x)\,dx
    \\ &=
    \sqrt{\pi}\frac{(2\lambda)_{2m}}{(2m)!}\frac{\Gamma(\beta)}{\Gamma(\beta+1/2)}
    {}_3F_2(-2m,2\lambda+2m,\beta;2\beta,\lambda+1/2;1)
    \\ &=
    \pi\frac{(2\lambda)_{2m}}{(2m)!}
    \frac{\Gamma(\beta)\Gamma(\lambda+1/2)\Gamma(\beta-\lambda+1/2)}
    {\Gamma(1/2-m)\Gamma(\lambda+m+1/2)\Gamma(\beta+m+1/2)\Gamma(\beta-\lambda-m+1/2)},
  \end{split}
  \end{equation*}
  where in last identity we have evaluated the hypergeometric function with the so-called Watson formula \cite[eq. 16.4.6, p. 406]{Olver}. Therefore, if we denote $\alpha=\lambda/2+(1-\sigma)/4$, we obtain that
  \begin{equation}\label{eq:Lem02_key}
  \int_{-1}^{1}(1-x^2)^{\alpha-1}C_{2m}^\lambda(x)\,dx=
  R_{\sigma,\lambda}\int_{-1}^{1}(1-x^2)^{\alpha+\sigma/2-1}C_{2m}^{\lambda}(x)\,dx,
  \end{equation}
  with
  \begin{multline*}
  R_{\sigma,\lambda}=
  \frac{\Gamma(\alpha)\Gamma(\alpha-\lambda+1/2)}{\Gamma(\alpha+\sigma/2)\Gamma(\alpha-\lambda+1/2+\sigma/2)}\\\times
  \frac{\Gamma(\alpha+m+1/2+\sigma/2)\Gamma(\alpha-\lambda-m+1/2+\sigma/2)}{
    \Gamma(\alpha+m+1/2)\Gamma(\alpha-\lambda-m+1/2)}.
  \end{multline*}
  In this way, if we prove the identity
  \begin{equation}\label{eq:Lem02_relation}
  R_{\sigma,\lambda}=Q_{\sigma,\lambda}^{-1}\frac{\Gamma(2m+2\alpha+\sigma)}{\Gamma(2m+2\alpha)}
  \end{equation}
  we will conclude the proof, because \eqref{eq:Lem02_key} implies
  \[
  a_n^\lambda\left(\frac{1}{(1-x^2)^{\alpha+\sigma/2}}\right)
  =Q_{\sigma,\lambda}^{-1}\frac{\Gamma(n+2\alpha+\sigma)}{\Gamma(n+2\alpha)}
  a_n^\lambda\left(\frac{1}{(1-x^2)^\alpha}\right),
  \]
  where we have had in mind that the $n$-th Fourier coefficient is null when $n=2m+1$.

  Let us check that \eqref{eq:Lem02_relation} actually holds. Using the reflection formula \cite[eq. 6.1.17, p. 256]{AS} twice we have
  \begin{align*}
  \frac{\Gamma(\alpha-\lambda-m+1/2+\sigma/2)}{\Gamma(\alpha-\lambda-m+1/2)}
  &=
  \frac{\Gamma(\alpha+m+\sigma/2)}{\Gamma(\alpha+m)}\frac{\sin(\pi(\alpha-\lambda-m+1/2))}
  {\sin(\pi(\alpha-\lambda-m+1/2+\sigma/2))} \\
  &=
  \frac{\Gamma(\alpha+m+\sigma/2)}{\Gamma(\alpha+m)}
  \frac{\Gamma(\alpha)\Gamma(\alpha-\lambda+1/2+\sigma/2)}{\Gamma(\alpha+\sigma/2)\Gamma(\alpha-\lambda+1/2)},
  \end{align*}
  and then
  \begin{align*}
  R_{\sigma,\lambda}
  &=
  \frac{\Gamma(\alpha)^2}{\Gamma(\alpha+\sigma/2)^2}
  \frac{\Gamma(\alpha+m+\sigma/2)\Gamma(\alpha+m+\sigma/2+1/2)}{\Gamma(\alpha+m)\Gamma(\alpha+m+1/2)}
  \\ &=
  Q_{\sigma,\lambda}^{-1}\frac{\Gamma(2m+2\alpha+\sigma)}{\Gamma(2m+2\alpha)},
  \end{align*}
  by the duplication formula \cite[eq. 6.1.18, p. 256]{AS}.
\end{proof}

\section{Proof of Theorem \ref{theorem01}}

Polarizing the identity \eqref{eq:esc-prod} in Lemma \ref{lemma01} we obtain
\begin{equation}\label{eq:Thm01_key}
  \langle g,A_\sigma^{\lambda} f\rangle_\lambda=\frac{1}{2}\int_{-1}^{1}\int_{-1}^{1}F(x,y)
  K_\sigma^\lambda(x,y)\,d\mu_\lambda(y)\,d\mu_\lambda(x)+E_{\sigma,\lambda}\langle g,f\rangle_\lambda,
\end{equation}
with $F(x,y)=(g(x)-g(y))(f(x)-f(y))$.

Let us take $g(x)=(1-x^2)^{-\lambda/2-(1-\sigma)/4}$ and $f(x)=u^2(x)/g(x)$ for $u\in H_\lambda^\sigma$. Then
$$
  F(x,y)=\left(u(x)-u(y)\right)^2-g(x)g(y)\left(\frac{u(x)}{g(x)}-\frac{u(y)}{g(y)}\right)^2
$$
and \eqref{eq:Thm01_key} becomes
\begin{multline*}
  \langle g,A_\sigma^{\lambda} f\rangle_\lambda\\
  =
  \langle u,A_\sigma^{\lambda} u\rangle_\lambda-\frac{1}{2}\int_{-1}^{1}\int_{-1}^{1}g(x)g(y)\left(\frac{u(x)}{g(x)}-\frac{u(y)}{g(y)}\right)^2 K_\sigma^\lambda(x,y)\,d\mu_\lambda(y)\,d\mu_\lambda(x).
\end{multline*}
Now, by \eqref{eq:Lem02_principal}, we have
\[
  \langle g,A_\sigma^{\lambda} f\rangle_\lambda
  =
  \langle A_\sigma^{\lambda} g,f\rangle_\lambda
  =
  Q_{\sigma,\lambda}\int_{-1}^{1}\frac{u^2(x)}{(1-x^2)^{\sigma/2}}\,d\mu_\lambda(x)
\]
and then we can deduce the ground state representation
\begin{multline}
\label{ec:ground}
\langle u,A_\sigma^{\lambda} u\rangle_\lambda-Q_{\sigma,\lambda}\int_{-1}^{1}\frac{u^2(x)}{(1-x^2)^{\sigma/2}}\,d\mu_\lambda(x)
\\=\frac{1}{2}\int_{-1}^{1}\int_{-1}^{1}g(x)g(y)\left(\frac{u(x)}{g(x)}-\frac{u(y)}{g(y)}\right)^2 K_\sigma^\lambda(x,y)\,d\mu_\lambda(y)\,d\mu_\lambda(x).
\end{multline}
So, due to the positivity of the kernel $K_{\sigma}^\lambda$, we conclude the proof.


\begin{thebibliography}{99}
\bibitem{AS}
  M. Abramowitz and I. A. Stegun (editors),
  \textit{Handbook of Mathematical Functions: With Formulas, Graphs, and Mathematical Tables},
  National Bureau of Standards Applied Mathematics Series, 55, Washington, 1964.

\bibitem{Beckner-Annals}
  W. Beckner,
  {Sharp Sobolev inequalities on the sphere and the Moser--Trudinger inequality},
  \textit{Ann. of Math. (2)} \textbf{138} (1993), 213--242.

\bibitem{Beckner3}
    W. Beckner,
    {Pitt's inequality and the uncertainty principle},
    \textit{Proc.  Amer.  Math.  Soc.}
    \textbf{123} (1995), 1897--1905.

\bibitem{Beckner2}
  W. Beckner.
  {Pitt's inequality with sharp convolution estimates},
  \textit{Proc. Amer. Math. Soc.}
  \textbf{136} (2008), 1871--1885.


\bibitem{Betancoretal}
   J. J. Betancor, J. C. Fariña, L. Rodríguez-Mesa, R. Testoni, and J. L. Torrea,
   {L. A choice of Sobolev spaces associated with ultraspherical expansions},
   \textit{Publ. Mat.}
   \textbf{54} (2010), 221--242.

\bibitem{Branson}
  T. P. Branson,
  Sharp inequalities, the functional determinant, and the complementary series,
  \textit{Trans. Amer. Math. Soc.} \textbf{347} (1995), 3671--3742.

\bibitem{CRT}
   \'O. Ciaurri, L. Roncal, and S. Thangavelu,
   {Hardy--type inequalities for fractional powers of the Dunkl--Hermite operator},
   \textit{preprint},  arXiv:1602.04997.

\bibitem{DaiXu}
  F. Dai and Y. Xu,
  \textit{Approximation theory and Harmonic Analysis on spheres and balls},
  Springer, New York, 2013.

\bibitem{FLS}
   R. L. Frank, E. H. Lieb, and R. Seiringer,
   {Hardy--Lieb--Thirring inequalities for fractional Schr\"odinger operators},
   \textit{J. Amer. Math. Soc.}
   \textbf{21} (2008), 925--950.

\bibitem{GL}
   N. Garofalo and E. Lanconelli,
   {Frequency functions on the Heisenberg group, the uncertainty principle and unique continuation},
   \textit{Ann. Inst. Fourier (Grenoble)}
   \textbf{40} (1990), 313--356.

\bibitem{GIT}
   D. V. Gorbachev, V. I. Ivanov, and S. Yu Tikhonov,
   {Sharp Pitt inequality and logarithmic uncertainty principle for Dunkl transform in $L^2$},
   \textit{J. Approx. Theory}
   \textbf{202} (2016), 109--118.

\bibitem{MS}
   B. Muckenhoupt and E. M. Stein,
   {Classical expansions and their relation to conjugate harmonic functions},
   \textit{Trans. Amer. Math. Soc.}
   \textbf{118} (1965), 17--92.

\bibitem{Olver}
  F. W. J. Olver (editor-in-chief),
  \textit{NIST Handbook of Mathematical Functions},
  Cambridge University Press, New York, 2010.

\bibitem{PrudnikovI}
  A. P. Prudnikov, Yu. A. Brychkov, and O. I. Marichev,
  \textit{Integrals and series}. Vol. 1. Elementary functions.
  Translated from the Russian and with a preface by N. M. Queen.
  Gordon and Breach Science Publishers, New York, 1986.

\bibitem{PrudnikovII}
  A. P. Prudnikov, Yu. A. Brychkov, and O. I. Marichev,
  \textit{Integrals and series}. Vol. 2. Special functions.
  Translated from the Russian by N. M. Queen.
  Gordon and Breach Science Publishers, New York, 1986.

\bibitem{RT}
  L. Roncal and S. Thangavelu,
  {Hardy's inequality for fractional powers of the sublaplacian on the Heisenberg group},
  \textit{Adv. Math.}
  \textbf{302} (2016), 106--158.

\bibitem{Rubin}
  B. Rubin,
  \textit{Introduction to Radon transforms. With elements of fractional calculus and harmonic analysis},
  Encyclopedia of Mathematics and its Applications, 160,
  Cambridge University Press, New York, 2015.

\bibitem{Samko}
  S. G. Samko,
  \textit{Hypersingular integrals and their applications},
  Analytical Methods and Special Functions, 5,
  Taylor \& Francis, London, 2002.

\bibitem{Sherman}
  T. O. Sherman,
  {The Helgason Fourier transform for compact Riemannian symmetric spaces of rank one},
  \textit{Acta Math.}
  \textbf{164} (1990), 73--144.

\bibitem{Yafaev}
  D. Yafaev,
  {Sharp constants in the Hardy--Rellich inequalities},
  \textit{J. Funct. Anal.}
  \textbf{168} (1999), 121--144.

\end{thebibliography}
\end{document}